\documentclass{amsart}

\usepackage{graphicx}

\newtheorem{theorem}{Theorem}[section]
\newtheorem{lemma}[theorem]{Lemma}
\newtheorem{proposition}[theorem]{Proposition}

\theoremstyle{definition}
\newtheorem{definition}[theorem]{Definition}
\newtheorem{example}[theorem]{Example}
\newtheorem{corollary}[theorem]{Corollary}

\theoremstyle{remark}
\newtheorem{remark}[theorem]{Remark}

\numberwithin{equation}{section}

\usepackage{ulem}
\usepackage{amsrefs}

\usepackage{color}
\definecolor{red}{rgb}{1,0,0}
\definecolor{green}{rgb}{0,1,0}
\definecolor{blue}{rgb}{0,0,1}
\definecolor{refkey}{gray}{.625}
\definecolor{labelkey}{gray}{.625}
\definecolor{fuck}{gray}{.5}

\let\oldmarginpar\marginpar
\renewcommand\marginpar[1]{\-\oldmarginpar[\raggedleft\footnotesize #1]%
{\raggedright\footnotesize #1}}

\usepackage[T1]{fontenc}
\usepackage{lmodern}
\usepackage[latin1]{inputenc}

\usepackage{setspace}
\usepackage{indentfirst}

\usepackage{multicol}
\usepackage{geometry}
\usepackage{hyperref}

\usepackage{amsmath,amssymb}
\usepackage{mathrsfs}

\usepackage{amsthm}

\theoremstyle{plain}

\newcommand{\C}{\mathbb{C}}

\newcommand{\R}{\mathbb{R}}

\newcommand{\g}{\mathfrak{g}}

\DeclareMathOperator{\id}{id}

\hypersetup{bookmarksopen=true,pdfpagelayout=SinglePage,colorlinks=false,pdfstartview=Fit,pdfpagemode=FullScreen}

\begin{document}

\title{Atiyah classes of Lie bialgebras}

\author{Wei Hong}
\address{School of Mathematics and Statistics, Wuhan University,Wuhan 430072, China}
\address{Hubei Key Laboratory of Computational Science (Wuhan University), Wuhan 430072, China}
\email{hong\textunderscore  w@whu.edu.cn}

\keywords{Atiyah class, Lie bialgebras, Lie algebra pair}

\begin{abstract}
The Atiyah class was originally introduced by M.F. Atiyah. It has many developments in recent years. One important case is the Atiyah classes of Lie algebra pairs.
In this paper, we study the Atiyah class of the Lie algebra pair associated with a Lie bialgebra $(\g,\g^*)$. A simple description of the Atiyah class and the first scalar Atiyah class is given by the Lie algebra structures on $\g$ and $\g^*$. As its application, the Atiyah classes for some special cases are investigated.
\end{abstract}

\maketitle


\section{Introduction}
The Atiyah class was originally introduced by M. F. Atiyah \cite{Atiyah} in order to describe the obstruction of the existence of a holomorphic connection on a holomorphic vector bundle.
In the late 1990's, Kontsevich \cite{Kontsevich 99} and Kapranov \cite{Kapranov 99} revealed the relation between Atiyah class and Rozansky-Witten invariants.
Subsequent works have appeared in many situations
\cite{Calaque-VandenBergh, C-C-T 13},  \cite{Bordermann 12},
\cite{C-S-X 2016, M-S-X 15, Batakidis-Voglaire 18} and etc. One interesting case is the Atiyah class associated with a Lie algebra pair $(L, A)$ and an $A$-module $E$. 
The geometric meaning of the Atiyah class of a Lie pair was studied in 
\cite{Wang 58, Nguyen 65}.
There are recent developments of the study of the Atiyah classes of Lie pairs
\cite{Bordermann 12, C-C-T 13, Laurent-Voglaire 15, C-S-X 2016}.

In this paper, we investigate the Atiyah class of the Lie pair associated with a Lie bialgebra $(\g,\g^*)$, or more precisely, the Lie algebra $L=\g\Join\g^*$, its subalgebra $A=\g$ and the 
$\g$-module $E=L/A\cong\g^*$.
Let us denote $F$ by the map 
\begin{equation*}
\g\otimes\g\xrightarrow{\id\otimes(-ad^*)}\g\otimes End(\g^*).
\end{equation*}
Then $F$ is a morphism between the $\g$-modules $\g\otimes\g$ and 
$\g\otimes End(\g^*)$. It induces a map 
\begin{equation*}
H^1(\g, \g\otimes\g)\xrightarrow{F_*} H^1(\g,\g\otimes End(\g^*)):~F_*(\alpha)(x)=F(\alpha(x)),
\end{equation*}
for all $\alpha\in H^1(\g, \g\otimes\g)$ and $x\in\g$. 
We have the following theorem for the Atiyah class $\alpha_E$ associated with the triple $(L=\g\bowtie\g^{*}, A=\g, E=\g^*)$. 

\begin{theorem}\label{thm-AtiyahLiebi3}
Let $(\g, \g^*)$ be a Lie bialgebra with the associated map $\gamma:\g\mapsto \g\otimes\g$.
Let $\lambda\in\g^*\otimes\g\otimes End(\g^*)$ be defined by $\lambda(x,\xi)=ad_{ad_{\xi}^*(x)}^*\in End(\g^*)$ for all $x\in\g$ and $\xi\in\g^*$. Then
\begin{enumerate}
\item The map $\lambda:\g\mapsto\g\otimes End(\g^*)$ satisfies
 $\lambda=-F\circ\gamma$.
\item The cohomology class $[\lambda]\in H^1(\g,\g\otimes End(\g^*)$ satisfies 
$[\lambda]=-F_*[\gamma]$.
\item the Atiyah class $\alpha_E=[\lambda]$ associated with the triple $(L=\g\bowtie\g^{*}, A=\g, E=\g^*)$
vanishes if and only if $[\gamma]\in\ker F_*$.
\end{enumerate}
\end{theorem}

Given a Lie algebroid pair $(L,A)$ and an $A$-module $E$, the scalar Atiyah classes  $c_k(E)$ is defined by Chen-Sti\'{e}non-Xu in \cite{C-S-X 2016}. 
Let $\kappa\in\g^*$ be the modular vector of the Lie algebra $\g$, defined by 
\begin{equation*}
\kappa(x)=tr(ad_x),\quad\forall x\in\g.
\end{equation*}
Let $\gamma:\g\rightarrow \g\otimes\g$ be the cocycle associated with the 
Lie bialgebra $(\g,\g^*)$.
Let the map $\imath_{\kappa}\gamma: \g\rightarrow\g$ be defined by
\begin{equation*}
(\imath_{\kappa}\gamma)(x)=\imath_{\kappa}\gamma(x),\quad\forall x\in\g,
\end{equation*}
where $\imath_{\kappa}\gamma(x)$ denotes by the contraction of $\kappa\in\g^*$ 
with the first part of $\gamma(x)\in\g\otimes\g$.
Then we have the following theorem for the first scalar Atiyah class $c_1(E)$ 
associated with the triple $(L=\g\bowtie\g^{*}, A=\g, E=\g^*)$.

\begin{theorem}\label{c1-thm}
Let $c_1(E)$ be the first scalar Atiyah class associated with the triple 
$(L=\g\bowtie\g^{*}, A=\g, E=\g^*)$. Then we have
\begin{enumerate}
\item
\begin{equation}
c_1(E)=-\frac{\sqrt{-1}}{2\pi}[\imath_{\kappa}\gamma].
\end{equation}
\item $c_1(E)$ vanishes if and only if there exists $v\in\g$ such that 
\begin{equation}\label{kappav-eqn2}
ad_{\kappa}^*=ad_v\in End(\g),
\end{equation}
where $ad_{\kappa}^*\in End(\g)$ is the dual map of $ad_{\kappa}\in End(\g^*)$.
\item The Equation \eqref{kappav-eqn2} is equivalent to 
\begin{equation}\label{kappav-eqn1}
ad_{\kappa+v}(\g)=0, 
\end{equation}
where $ad_{\kappa+v}$ is considered an element in $End(L)$, and $\g$ is considered as a subspace of $L$.
\end{enumerate}
\end{theorem}

In \cite{Lu-Weinstein 90}, it is shown that $(L=sl(n,\C), \g=su(n),\g^*=sb(n,\C))$ and
$(L=sl(n,\C), \g=sb(n,\C),\g^*=su(n))$ are Manin triples. We investigate the Atiyah classes for both situations. In Example \ref{Lu-Weinstein-exa1}, we show that the Atiyah class associated with the triple $(L=sl(n,\C), A=\g=su(n),E=\g^*=sb(n,\C))$ does vanish. By contrast, in Proposition \ref{LW-prop}, we prove that the Atiyah class associated with the triple $(L=sl(n,\C), A=\g=sb(n,\C),E=\g^*=su(n))$ does not vanish.

{\bf Ackowledgements}  We would like to thank Zhanqiang Bai, Zhuo Chen, Camille Laurent-Gengoux, Zhangju Liu, Yu Qiao, Yannick Voglaire and Ping Xu for helpful discussions and comments. Hong's research is partially supported by NSFC grant 11401441. 

\section{Preliminary}
\subsection{Atiyah classes for Lie algebroid pairs}
In \cite{C-S-X 2016},  Chen, Sti\'{e}non and Xu  introduced the Atiyah class for Lie algebroid pairs. A Lie algebroid pair $(L, A)$ is a Lie algebroid $L$ together with a Lie subalgebroid $A$ over the same base manifold.
Assume that $E$ is an $A$-module, and $\nabla$ is an $L$-connection on $E$ extending its $A$-action. The curvature of $\nabla$ is the bundle map 
$R^{\nabla}_E: \wedge^{2}L\rightarrow End(E)$ defined by
\begin{equation}\label{Liepaircur-eqn1}
R_E^{\nabla}(l_1,l_2)=\nabla_{l_1}\nabla_{l_2}-\nabla_{l_2}\nabla_{l_1}-\nabla_{[l_1,l_2]}
\end{equation}
for all $l_1,l_2\in\Gamma(L)$.
Since $E$ is an $A$-module, the restriction of $R^{\nabla}_E$ to $\wedge^2 A$ vanishes. Hence the curvature induces a
section $R^{\nabla}_E\in\Gamma(A^*\otimes A^{\perp}\otimes End(E))$, or equivalently, a bundle map
$R_E^{\nabla}:A\otimes(L/A)\rightarrow End(E)$ given by
\begin{equation}\label{Liepaircur-eqn2}
R_E^\nabla(a,\bar{l})=\nabla_{a}\nabla_{l}-\nabla_{l}\nabla_{a}-\nabla_{[a,l]}
\end{equation}
for all $a\in \Gamma(A)$ and $l\in \Gamma(L)$.
The $L$-connection $\nabla$ is compatible with the A-module structure on E if and only if $R^{\nabla}_E=0$.
\begin{theorem} \label{C-S-X}\cite{C-S-X 2016}
\begin{enumerate}
\item The section $R_E^\nabla$ of $A^*\otimes A^{\perp}\otimes End(E)$ is a 1-cocycle for Lie algebroid $A$ with values in the $A$-module $A^{\perp}\otimes End(E)$. We call $R_E^\nabla$ the Atiyah cocycle associated with the $L$-conncetion $\nabla$ that extends the $A$-module structure of $E$.
\item The cohomology class $\alpha_{E}\in H^{1}(A, A^{\perp}\otimes End(E) )$ 
of the cocycle $R_E^\nabla$ does not depend on the choice of the
$L$-connetion extending the $A$-action. 
And the cohomology class $\alpha_{E}\in H^{1}(A, A^{\perp}\otimes End(E) )$  is called the Atiyah class of the $A$-module $E$.
\item  The Atiyah class $\alpha_E$ of $E$ vanishes if and only if there exists an $A$-compatible $L$-connection on $E$.
\end{enumerate}
\end{theorem}

Given a Lie algebroid pair $(L,A)$ and an $A$-module $E$, the scalar Atiyah classes  $c_k(E)$ is defined \cite{C-S-X 2016} by 
\begin{equation}
c_k(E)=\frac{1}{k!}(\frac{\sqrt{-1}}{2\pi})^k tr(\alpha_E^k)\in H^k(A,\wedge^k A^{\perp}).
\end{equation}
Here $\alpha_E^k$ denotes the image of $\alpha_E\otimes\cdots\otimes\alpha_E$ under the map 
$$H^1(A, A^\perp\otimes End(E))\wedge\cdots\wedge 
H^1(A, A^\perp\otimes End(E)),$$
which is induced by the composition in $End(E)$ and the wedge product in 
$\wedge^{\bullet}A^{\perp}$.

Let $(L, A)$ be a Lie algebroid pair. Then $E=L/A$ naturally becomes an $A$-module, with the $A$-modules structure on $E=L/A$ defined by 
$$a\cdot\overline{l}=\overline{[a,l]}$$
for all $a\in\Gamma(A)$ and $l\in\Gamma(L)$. 
In the special case of $(L, A)$ being a Lie algebra pair, we can define
the Atiyah class associated with $(L, A, L/A)$ by Theorem \ref{C-S-X}.

\subsection{Lie algebra modules and Lie bialgebras}
We first recall some necessary knowledge of Lie algebras.
Let $\g$ be a Lie algebra over the field $\mathbf{k}=\R$ or $\C$.
A representation of $\g$ on a $\mathbf{k}$-vector space $V$ is a morphism $\rho:\g\rightarrow End{V}$ satisfying $$\rho([x,y])=[\rho(x),\rho(y)]$$
 for all $x,y\in\g$.
The action map of $\g$ on $V$
$$\g\times V\rightarrow V: (x,v)\rightarrow x\cdot v=\rho(x)v, \quad\forall x\in\g,v\in V$$ 
gives a $\g$-module structure on $V$.

Suppose that $V, W$ are $\g$-modules with the associated representation $\rho_V$ and $\rho_W$. 
Then $V^*$, $End(V)$ and $V\otimes W$
are all $\g$-modules, with the corresponding representation given by
\begin{align*}
\rho_{V^*}=-\rho_{V}^*,\\
\rho_{End(V)}=[\rho_V,\cdot],\\
\rho_{V\otimes W}=\rho_V\otimes\id+\id\otimes\rho_W.
\end{align*}

The Lie algebra $\g$ acts on itself by the adjoint action: $x\in\g\mapsto ad_x\in End(\g),$
where $ad_x(y)=[x,y]$ for all $y\in\g$. The Lie algebra $\g$ acts on $\g^*$ by the coadjoint action:
$$x\in\g\mapsto -ad_x^*\in End(\g^*).$$ 
For a given Lie algebra $\g$, the vector space $\g\otimes\g$ and $\g\otimes End(\g^*)$ are both $\g$-modules,
with the $\g$-module structures given by
\begin{align}
x\cdot(y\otimes z)=ad_x(y)\otimes z+y\otimes ad_x(z),\label{g-module-eq1} \\
x\cdot(y\otimes T)=ad_x(y)\otimes T+y\otimes[-ad_{x}^*,T]\label{g-module-eq2}
\end{align}
for all $x,y,z\in\g$ and $T\in End(\g^*)$.
The action of Lie algebra $\g$ on $\g\otimes\g$ in Equation \eqref{g-module-eq1} is also called the adjoint representation, denoted by $ad^{(2)}: \g\mapsto End(\g\otimes\g)$:
\begin{equation}
x\rightarrow ad_x\otimes\id+\id\otimes ad_x
\end{equation}
for all $x\in\g$.

{\bf In this paper, we take the $\g$-module structures above on the corresponding spaces
without special explanation.}

Given a $\g$-module $V$, the Lie algebra cohomology $H^{*}(\g,V)$ is defined by the 
Chevalley-Eilenberg complex.
The coboundry of $f\in Hom(\wedge^k\g, V)$
is an element in $\delta f\in Hom(\wedge^{k+1}\g, V)$, given by
\begin{align*}
(\delta f)(x_0,x_1,\ldots,x_n)=\sum_{i=0}^{k}(-1)^{i}\rho(x_i)f(x_0,\ldots,\hat{x_i},\ldots,x_k)\\
+\sum_{0\leq i<j\leq n}(-1)^{i+j}f([x_i,x_j],x_0,\ldots,\hat{x_i},\ldots,\hat{x_j},\ldots,x_k),
\end{align*}
for $x_0, x_1,\ldots, x_k\in\g$.

Next we will recall some classical theory of Lie bialgebras (see 
\cite{Kosmann-Schwarzbach1997}).
\begin{definition}
A Lie bialgebra is a Lie algebra $\g$ with a linear map $\gamma:\g\rightarrow\g\otimes\g$ such that
\begin{enumerate}
\item the dual map $\gamma^t: \g^*\otimes\g^*\rightarrow\g^*$ defines a Lie bracket on $\g^*$, i.e.,
is a skew-symmetric bi-linear map satisfying the Jacobi identity, and
\item $\gamma$ is a cocycle on $\g$ with values in $\g\otimes\g$, where $\g$ acts on $\g\otimes\g$ by
the adjoint representation $ad^{(2)}$.
\end{enumerate}
\end{definition}

The cocycle $\gamma:\g\rightarrow\g\otimes\g$ defines a Lie bracket on $\g^*$, determined by
$$\langle[\xi,\eta],x\rangle=\langle\gamma(x),\xi\otimes\eta\rangle,$$
for $\xi,\eta\in\g^*$ and $x\in\g$. In fact, $\gamma$ is a linear map from 
$\g$ to $\g\wedge\g$.
The map $\gamma:\g\mapsto \g\otimes\g$ is a cocyle, thus $\gamma$ defines a cohomology class $[\gamma]\in H^1(\g,\g\otimes\g)$.
If the cocycle $\gamma$ is a coboundry, i.e., $\gamma=\delta r$, 
with $r\in\g\otimes\g$, the corresponding Lie bialgebra
is called a \emph{coboundary Lie bialgebra}, 
and $r\in\g\otimes\g$ is called a \emph{$r$-matrix}.

The following is an equivalent definition of Lie bialgebra.
\begin{definition}
A Lie bialgebra consist of a pair of vector spaces $(\g,\g^*)$, such that
\begin{enumerate}
\item $\g$ and $\g^*$ are both Lie algebras,
\item the vector space $\g\oplus\g^*$ is a quadratic Lie algebra, 
with the non-degenerate bi-linear form on $\g\oplus\g^*$
defined by $\langle x+\xi,y+\eta\rangle=\langle x,\eta\rangle+\langle\xi,y\rangle$
for all $x,y\in\g$ and $\xi,\eta\in\g^*$.
\item $\g$ and $\g^*$ are Lie subalgebras of the Lie algebra $\g\oplus\g^*$.
\end{enumerate}
\end{definition}
The Lie algebra $\g\oplus\g^*$ is called the double of $\g$ and $\g^*$, 
denoted by $\g\bowtie\g^*$.
The bracket between $\g$ and $\g^*$ is given by
\begin{equation*}
[x,\xi]=-ad_{x}^*\xi+ad_{\xi}^*x
\end{equation*}
for all $x\in\g$ and $\xi\in\g^*$. The triple $(L=\g\Join\g^*, \g, \g^*)$ is called a \emph{Manin triple}.

\begin{example}\label{Lu-Weinstein-exa}
\cite{Lu-Weinstein 90} Let $sb(n,\C)$ be the Lie algebra consisting of all $n\times n$ traceless upper 
triangular complex matrices with real diagonal elements. Then we have 
$$sl(n,\C)=su(n)\Join sb(n,\C).$$
If we define a non-degenerate bi-linear form on $sl(n,\C)$ by
\begin{equation} \label{Im-trace-eqn}
\langle X,Y\rangle=Im(trace(XY))
\end{equation}
for all $X,Y\in sl(n,\C)$, then $sl(n,\C)$ becomes a quadratic Lie algebra, 
and $su(n)$ and $sb(n,\C)$ are maximal isotropic subspaces of $L$. 
The pairing between $\g=sb(n,\C)$ and $\g^*=su(n)$ is defined by
$$\langle x,\xi\rangle=Im(trace(x\cdot\xi))$$
for all $x\in\g=sb(n,\C)$ and $\xi\in\g^*=su(n)$.
Thus $(L=sl(n,\C), \g=su(n),\g^*=sb(n,\C))$ and
$(L=sl(n,\C), \g=sb(n,\C),\g^*=su(n))$ are both Manin triples.
Moreover, $(\g=su(n), \g^*=sb(n,\C))$ is a coboundary Lie bialgebra.
\end{example}

\section{Atiyah classes of Lie bialgebras}

\subsection{Atiyah class associated with the triple $(L=\g\bowtie\g^{*}, A=\g, E=\g^*)$}
Let $(\g, \g^*)$ be a Lie bialgebra.
Let $L=\g\bowtie\g^{*}$, $ A=\g$ and $E=\g^*\simeq L/A$. 
Then $(L,A)$ is a Lie pair, $E$ is an $A$-module. The $A$-action on $E=\g^*$ is the coadjoint action, and the $A$-action on $A^{\perp}\simeq(L/A)^*\simeq\g$ is the adjoint action.  
Let $\nabla: L\mapsto End(E)$ be an $A$-compatible $L$-connection on $E$. 
The map $\nabla: L\mapsto End(E)$ splits into two parts: 
$$\nabla |_{\g}: \g\mapsto End(\g^*) \quad \text{and} 
\quad \nabla |_{\g^*}: \g^*\mapsto End(\g^*),$$ 
where $\nabla |_{\g}: \g\mapsto End(\g^*) $ is exactly the coadjoint action of 
$\g$ on $\g^*$. 
Let us denote the linear map $\nabla |_{\g^*}: \g^*\mapsto End(\g^*)$ by $S$. 
Then  $R_{E}^{\nabla}: A\otimes L/A\mapsto End(E)$  becomes 
$R_{E}^{\nabla}: \g\otimes\g^*\mapsto End(\g^*)$.
For all $x\in\g$ and $\xi\in\g^*$, recall that $[x,\xi]=-ad_x^*\xi+ad_{\xi}^*x$.
 By Equation \eqref{Liepaircur-eqn2}, the curvature
$R_{E}^{\nabla}: \g\otimes\g^*\mapsto End(\g^*)$ can be written as
\begin{equation}\label{eqn-REnabla}
R_{E}^{\nabla}(x,\xi)=-ad_{x}^*S(\xi)+S(\xi)ad_{x}^*+S(ad_{x}^*(\xi))+ad_{ad_{\xi}^*(x)}^*.
\end{equation}
Applying Theorem \ref{C-S-X} in this case, we obtain

\begin{theorem}\label{thm-AtiyahLiebi}
\begin{enumerate}
\item The element $R_{E}^{\nabla}\in\g^*\otimes\g\otimes End(\g^*)$ is a 1-cocycle for the Lie algebra $\g$ with values in the $\g$-module $\g\otimes End(\g^*)$;
\item the corresponding cohomology class $\alpha_{E}\in H^{1}(\g, \g\otimes End(\g^*) )$, called the  Atiyah class associated with the triple $(L=\g\bowtie\g^{*}, A=\g, E=\g^*)$, does not depend on the linear map $S: \g^*\mapsto End(\g^*) $;
\item the Atiyah class $\alpha_E$ vanishes if and only if there exists a linear map $S: \g^*\mapsto End(\g^*) $ such that $R_{E}^{\nabla}=0$, or in the other words,
\begin{equation}\label{Eq-Atiyahclass}
-ad_{x}^*S(\xi)+S(\xi)ad_{x}^*+S(ad_{x}^*(\xi))+ad_{ad_{\xi}^*(x)}^*=0
\end{equation}
for all $x\in\g$ and $\xi\in\g^*$.
\end{enumerate}
\end{theorem}

In Theorem \ref{thm-AtiyahLiebi}, the $\g$-module structure on $\g\otimes End(\g^*)$ is defined by Equation \eqref{g-module-eq2}.

The following corollary is an immediate consequence of Theorem \ref{thm-AtiyahLiebi}.
\begin{corollary}\label{cor-Liebialgebra1}
Let $(\g, \g^*)$ be a Lie bialgebra. Then the following conditions are equivalent:
\begin{itemize}
\item[(a)] the annihilator of $Center(\g)$ is  not an ideal of $\g^*$;
\item[(b)] $Center(\g)$ is not an invariant subspace of $\g$ under the coadjoint action of $\g^*$ on $\g\cong(\g^*)^*$;
\item[(c)] there exists $x\in Center(\g)$ and $\xi\in\g^*$ satisfying $ad_{\xi}^*(x)\not\in Center(\g)$. 
\end{itemize}
 If one of the above equivalent conditions is satisfied, then the Atiyah class 
 $\alpha_E$ associated with the triple $(L=\g\bowtie\g^{*}, A=\g, E=\g^*)$ does not vanish.
\end{corollary}

\begin{proof}
\begin{enumerate}
\item We first prove the equivalence of the conditions.
\begin{itemize}
\item $(b)\Leftrightarrow(c)$ The equivalence of $(b)$ and $(c)$ is obvious.
\item $(a)\Rightarrow(c)$ If $Center(\g)^{\perp}$ is  not an ideal of $\g^*$,
then there exist $\xi\in\g^*$ and $\eta\in (Center(\g))^{\perp}\subset\g^*$, such that
$[\xi,\eta]\notin (Center(\g))^{\perp}$. Hence there exists $x\in\g$, such that
$\langle x,[\xi,\eta]\rangle\neq 0,$
which implies that
 $\langle ad_{\xi}^*(x),\eta\rangle=\langle x,[\xi,\eta]\rangle\neq 0$. 
 Since $\eta\in (Center(\g))^{\perp}$, we obtain that $ad_{\xi}^*(x)\not\in Center(\g)$.
\item $(c)\Rightarrow(a)$ The proof is similar as above. We skip it.
\end{itemize}
\item If one of the equivalent conditions is satisfied, i.e., the last condition is satisfied,
there exists $x\in Center(\g)$ and $\xi\in\g^*$ satisfying $ad_{\xi}^*(x)\not\in Center(\g)$.
The condition $x\in Center(\g)$ implies that $ad_{x}^*=0$.
And the condition $ad_{\xi}^*(x)\not\in Center(\g)$ implies that $ad_{ad_{\xi}^*(x)}^*\neq 0$. 
Therefore it does not exist the map $S: \g^*\mapsto End(\g^*) $ satisfying Equation \eqref{Eq-Atiyahclass}. 
By Theorem \ref{thm-AtiyahLiebi}, the Atiyah class $\alpha_E$ associated with the triple 
$(L=\g\bowtie\g^{*}, A=\g, E=\g^*)$ does not vanish.
\end{enumerate}
\end{proof}

\begin{example}(An example with non-vanishing Atiyah class)
Let $(\g, \g^*)$ be a 3-dimensional Lie bialgebra (see \cite{Hong-Liu 09}), 
with the Lie brackets on $\g$ and $\g^*$ being defined as
\begin{align*}
&[x_1,x_2]=x_3, \quad [x_2, x_3]=0,  \quad [x_3, x_1]=0;\\
&[\xi^1,\xi^2]=\xi^2, \quad [\xi^2, \xi^3]=0,  \quad [\xi^3, \xi^1]=-\xi^3;
\end{align*}
where $\{x_1, x_2, x_3\}$ is a basis of $\g$,  and $\{\xi^1, \xi^2, \xi^3\}$ is the dual basis of $\g^*$. The center of Lie algebra $\g$ is spanned by $x_3$.  
As $ad_{\xi^3}^*(x_3)=-x_1\not\in Center(\g)$, by Corollary \ref{cor-Liebialgebra1}, 
the Atiyah class $\alpha_E$ associated with the triple $(L=\g\bowtie\g^{*}, A=\g, E=\g^*)$ does not vanish.
\end{example}

Choosing $S=0$ in Theorem \ref{thm-AtiyahLiebi}, we get another version of Theorem \ref{thm-AtiyahLiebi}.

\begin{theorem}\label{thm-AtiyahLiebi2}
Let $(\g,\g^*)$ be a Lie bialgebra. Let $\lambda\in\g^*\otimes\g\otimes End(\g^*)$ be defined by $\lambda(x,\xi)=ad_{ad_{\xi}^*(x)}^*\in End(\g^*)$, 
for all $x\in\g$ and $\xi\in\g^*$.  Then
\begin{enumerate}
\item $\lambda\in\g^*\otimes\g\otimes End(\g^*)$ is a 1-cocycle for the Lie algebra $\g$ with values in
the $\g$-module $\g\otimes End(\g^*)$,
\item the cohomology class $[\lambda]\in H^1(\g,\g\otimes End(\g^*))$ coincide with the Atiyah class $\alpha_E$ associated with the triple $(L=\g\bowtie\g^{*}, A=\g, E=\g^*)$.
\item the cohomology class $[\lambda]$ vanishes if and only if there exists a linear map
$S: \g^*\mapsto End(\g^*)$ such that
\begin{equation}\label{eqn-Lambda}
\lambda(x,\xi)=ad_{x}^*\cdot S(\xi)-S(\xi)\cdot ad_{x}^*-S(ad_{x}^*(\xi)),
\end{equation}
for all $x\in\g$ and $\xi\in\g^*$.
\end{enumerate}
\end{theorem}

\begin{remark}
If we consider $S: \g^*\mapsto End(\g^*)$ as an element in $\g\otimes End(\g^*)$,
 Equation \eqref{eqn-Lambda} can then be written as the coboundary condition
\begin{equation}
\lambda(x)=-x\cdot S,
\end{equation}
where the action of $x\in\g$ on $S\in\g\otimes End(\g^*)$ is defined in Equation \eqref{g-module-eq2}.
\end{remark}

Let us denote $F$ by the map 
\begin{equation}
\g\otimes\g\xrightarrow{\id\otimes(-ad^*)}\g\otimes End(\g^*).
\end{equation}
It is easy to verify that $F$ is a morphism between the $\g$-modules $\g\otimes\g$ and 
$\g\otimes End(\g^*)$. Thus $F$ induces a map 
\begin{equation}
H^1(\g, \g\otimes\g)\xrightarrow{F_*} H^1(\g,\g\otimes End(\g^*)):~F_*(\alpha)(x)=F(\alpha(x)),
\end{equation}
for all $\alpha\in H^1(\g, \g\otimes\g)$ and $x\in\g$.


{\bf Proof of Theorem \ref{thm-AtiyahLiebi3}:}
\begin{proof}
For any $x\in\g$ and $\xi\in\g^*$, let us denote $\imath_{\xi}\gamma(x)$ by the 
contraction of $\xi\in\g^*$ with the first part of $\gamma(x)\in\g\otimes\g$. 
For any $\eta\in\g^*$, we have
\begin{align*}
\langle\imath_{\xi}\gamma(x),\eta \rangle
&=\langle x,[\xi,\eta]\rangle\\
&=\langle [x,\xi],\eta\rangle\qquad\text{(by the invariant product)}\\
&=\langle -ad_x^*\xi+ad_{\xi}^*x,\eta\rangle\\
&=\langle ad_{\xi}^*x,\eta\rangle.
\end{align*}
Thus we get
\begin{equation}\label{ixigammma-eqn}
\imath_{\xi}\gamma(x)=ad_{\xi}^*x.
\end{equation}
As $F$ is defined by $\g\otimes\g\xrightarrow{\id\otimes(-ad^*)}\g\otimes End(\g^*)$, 
$F\circ\gamma(x)$ is an element in $\g\otimes End(\g^*)$. 
Let us denote $\imath_\xi (F\circ\gamma(x))$ by the contraction of $\xi\in\g^*$ with the first part of $F\circ\gamma(x)\in\g\otimes End(\g^*)$.
Then we have
\begin{align*}
\imath_\xi (F\circ\gamma(x))&=\imath_{\xi}((\id\otimes(-ad^*))\gamma(x))\\
&=(\imath_{\xi}\otimes(-ad^*))\gamma(x)\\
&=-ad_{\imath_{\xi}\gamma(x)}^*.
\end{align*}
By Equation \eqref{ixigammma-eqn}, it follows
\begin{align*}
\imath_\xi (F\circ\gamma(x))=-ad_{\imath_{\xi}\gamma(x)}^*=-ad_{ad_{\xi}^*x}^*
&=-\lambda(x,\xi).
\end{align*}
Thus we obtain
$$\lambda=-F\circ\gamma.$$
As $F$ is a morphism of $\g$-modules, we get
$$[\lambda]=-F_*[\gamma].$$
By theorem \ref{thm-AtiyahLiebi}, the Atiyah class $\alpha_E=[\lambda]$. It vanishes if and only if $[\gamma]\in\ker F_*$.
\end{proof}

\begin{remark}
Notice that in Theorem \ref{thm-AtiyahLiebi3} , 
the map $\gamma$ is only related to the Lie algebra structures on $\g^*$, 
and $F$ is only related to the Lie algebra structures on $\g$.
\end{remark}

\begin{corollary}\label{coboundary-cor}
Let $(\g, \g^*)$ be a coboundry Lie bialgebra with the $r$-matrix, i.e., $\gamma=\delta r$, where $\gamma:\g\rightarrow \g\otimes\g$ is the cocycle associated with $(\g, \g^*)$ and $r\in\g\otimes\g$.
\begin{enumerate}
\item 
The Atiyah class $\alpha_E$ associated with the triple $(L=\g\bowtie\g^{*}, A=\g, E=\g^*)$ vanishes.
\item Let the map $S: \g^*\rightarrow End(\g^*)$ be defined by 
\begin{equation}
S(\xi)=-ad_{r(\xi)}^*
\end{equation}
for all $\xi\in\g^*$, where $r(\xi)\in\g$ denotes by the contraction of $\xi$ with the first part of $r\in\g\otimes\g$. Then $S$ satisfies the Equation \eqref{eqn-Lambda}.
\end{enumerate}
\end{corollary}
\begin{remark}
The first part of the Corollary \ref{coboundary-cor} is due to K. Abdeljellil and
 Camille Laurent-Gengoux by private communication.
\end{remark}

\begin{proof}
\begin{enumerate}
\item For a coboundry Lie bialgebra $(\g,\g^*)$, the corresponding cohomology class 
$[\gamma]=0$. By theorem \ref{thm-AtiyahLiebi3}, the Atiyah class $\alpha_E$ associated with the triple $(L=\g\bowtie\g^{*}, A=\g, E=\g^*)$ vanishes.
\item By Theorem \ref{thm-AtiyahLiebi3}, we have
$\lambda=-F\circ\gamma=-F\circ\delta(r)$.

As $F: \g\otimes\g\xrightarrow{\id\otimes(-ad^*)}\g\otimes End(\g^*)$ is a morphism of 
$\g$-modules, where the $\g$-module structures are defined in Equation \eqref{g-module-eq1} and Equation \eqref{g-module-eq2}, we get 
$$F\circ\delta(r)=\delta (F(r)).$$
As a consequence, we have $\lambda=-\delta (F(r))$.

On the other hand, we have 
$$S(\xi)=-ad_{r(\xi)}^*=\imath_{\xi}(\id\otimes (-ad^*)(r))=\imath_{\xi}F(r),$$ 
which implies that 
$$S=F(r),$$ 
where $S:\g^*\rightarrow End(\g^*)$ is considered as an element in $\g\otimes End(\g^*)$.

Thus we have
$$\lambda=-\delta S,$$
which is equivalent to the Equation \eqref{eqn-Lambda}.
\end{enumerate}
\end{proof}

\begin{example}\label{Lu-Weinstein-exa1}
As shown in example \ref{Lu-Weinstein-exa}, $(\g=su(n),\g^*=sb(n,\C))$ is a coboundary Lie bialgebra. Hence by Corollary \ref{coboundary-cor}, the Atiyah class associated with
 $(L=\g\Join\g^*=sl(n,\C), \g=su(n), E=\g^*=sb(n,\C))$ vanishes.
\end{example}

\subsection{The first scalar Atiyah class associated with the triple $(L=\g\bowtie\g^{*}, A=\g, E=\g^*)$}
Let $(\g,\g^*)$ be a Lie bialgebra and let $\alpha_E\in H^1(\g, \g\otimes End(\g^*))$ 
be the Atiyah class associated with the triple 
$(L=\g\bowtie\g^{*}, A=\g, E=\g^*)$.  The map 
$$tr:\g\otimes End(\g^*)\xrightarrow{\id\otimes tr}\g\otimes k=\g$$
is a morphism of $\g$-modules.
The first scalar Atiyah class $c_1(E)$ \cite{C-S-X 2016} is defined by
\begin{equation}
c_1(E)=\frac{\sqrt{-1}}{2\pi}tr(\alpha_E)\in H^1(\g, \g).
\end{equation}

Let $\kappa\in\g^*$ be the modular vector of Lie algebra $\g$, defined by 
\begin{equation}\label{kappa-eqn}
\kappa(x)=tr(ad_x),\quad\forall x\in\g.
\end{equation}
Let $\gamma:\g\rightarrow \g\otimes\g$ be the cocycle associated with the 
Lie bialgebra $(\g,\g^*)$.
We define the map $\imath_{\kappa}\gamma: \g\rightarrow\g$ by
\begin{equation}
(\imath_{\kappa}\gamma)(x)=\imath_{\kappa}\gamma(x),\quad\forall x\in\g,
\end{equation}
where $\imath_{\kappa}\gamma(x)$ denotes by the contraction of $\kappa\in\g^*$ 
with the first part of $\gamma(x)\in\g\otimes\g$.

\begin{lemma}\label{lem-kappagamma}
The map $\imath_{\kappa}\gamma: \g\rightarrow\g$ is a cocycle for the Lie algebra $\g$ with values in $\g$. 
\end{lemma}
\begin{proof}
The map $\imath_{\kappa}\gamma: \g\rightarrow\g$ is the composition of 
the map $\gamma:\g\rightarrow \g\otimes\g$ and the map 
$\g\otimes\g\xrightarrow{\kappa\otimes\id}k\otimes\g=\g$. 
As $\gamma:\g\rightarrow \g\otimes\g$ is a cocycle, we only need to verify that the map 
$\g\otimes\g\xrightarrow{\kappa\otimes\id}k\otimes\g=\g$ is a morphism of 
between the $\g$-modules $\g\otimes\g$ and $\g$.

For any $x,y,z\in\g$, we have
\begin{align*}
(\kappa\otimes\id)(x\cdot(y\otimes z) )&=\kappa([x,y])z+\kappa(y)[x,z]\\
&=tr(ad_{[x,y]})z+tr(ad_y)[x,z]\\
&=tr(ad_y)[x,z]
\end{align*}
and 
\begin{align*}
&x\cdot((\kappa\otimes\id)(y\otimes z))\\
&=x\cdot(\kappa(y)z)=\kappa(y)[x,z]\\
&=tr(ad_y)[x,z].
\end{align*}
It proves that $\g\otimes\g\xrightarrow{\kappa\otimes\id}k\otimes\g=\g$ is a morphism of 
between the $\g$-modules. And consequently, the map $\imath_{\kappa}\gamma: \g\rightarrow\g$ is a cocycle.
\end{proof}

By Lemma \ref{lem-kappagamma}, the map $\imath_{\kappa}\gamma$ defines a cohomology class $[\imath_{\kappa}\gamma]\in H^1(\g,\g)$.

{\bf Proof of Theorem \ref{c1-thm}:}
\begin{proof}
\begin{enumerate}
\item By Theorem \ref{thm-AtiyahLiebi3}, we obtain 
$\alpha_E=[\lambda]$ and
$\lambda=-F\circ\gamma,$ 
where $F$ is the map 
 $$\g\otimes\g\xrightarrow{\id\otimes(-ad^*)}\g\otimes End(\g^*).$$
 Thus we have
 \begin{equation}\label{thmc1-eq1}
tr(\lambda)=-(\id\otimes tr(-ad^*))\circ\gamma=(\id\otimes tr(ad^*))\circ\gamma.
\end{equation}
As $$tr(ad^*_x)=tr(ad_x)=\kappa(x)$$ for all $x\in\g$, by Equation \eqref{thmc1-eq1}
we get that
\begin{equation}\label{thmc1-eq2}
tr(\lambda)=(\id\otimes\kappa)\circ\gamma.
\end{equation}
On the other hand, $\gamma$ is a map from $\g$ to $\g\wedge\g\subset\g\otimes\g$, 
which implies
\begin{equation} \label{thmc1-eq3}
\imath_{\kappa}\gamma=-(\id\otimes\kappa)\circ\gamma.
 \end{equation}
By Equations \eqref{thmc1-eq2} and \eqref{thmc1-eq3}, we obtain
$$c_1(E)=\frac{\sqrt{-1}}{2\pi}[tr(\gamma)]=-\frac{\sqrt{-1}}{2\pi}
[\imath_{\kappa}\gamma].$$
\item
By the arguments above, $c_1(E)$ vanishes if and only if
$$[\imath_{\kappa}\gamma]=0$$ in $H^1(\g,\g)$, or equivalently, 
there exist $v\in\g$ such that
\begin{equation}\label{c1-thm-eqn4}
\imath_{\kappa}\gamma(x)=ad_v(x)
\end{equation}
for all $x\in\g$.
The Equation \eqref{c1-thm-eqn4} is equivalent to
\begin{equation}\label{c1-thm-eqn5}
\langle \imath_{\kappa}\gamma(x),\eta\rangle=\langle ad_v(x),\eta\rangle
\end{equation}
for all $x\in\g$ and $\eta\in\g^*$.
The left side of Equation \eqref{c1-thm-eqn5} can be written as
\begin{gather*}
\langle \imath_{\kappa}\gamma(x),\eta\rangle=
\langle\gamma(x),\kappa\otimes\eta\rangle=\langle x,[\kappa,\eta]\rangle\\
=\langle x,ad_{\kappa}\eta\rangle=\langle ad_{\kappa}^*x,\eta\rangle.
\end{gather*}
Thus the Equation \eqref{c1-thm-eqn5} holds if and only if
$$ad_{\kappa}^*x=ad_v(x)$$
for all $x\in\g$.
Therefore $c_1(E)$ vanishes if and only if there exists $v\in\g$ such that 
$$ad_{\kappa}^*=ad_v.$$
\item
For all $y\in\g$, we have
$$\langle ad_x^*\kappa, y\rangle=\langle\kappa, [x,y]\rangle=trace(ad_{[x,y]})=trace([ad_x,ad_y])=0.$$
Thus we obtain that
\begin{equation}\label{kappav-eqn3}
ad_x^*\kappa=0
\end{equation}
for all $x\in\g$. 
For any $x\in\g$ and $\xi\in\g^*$, we have
\begin{equation*}
[\kappa+v, x]=-ad_{\kappa}^*x+ad_x^*\kappa+[v, x]
=(ad_v-ad_{\kappa}^*)x+ad_x^*\kappa.
\end{equation*}
By Equation \eqref{kappav-eqn3}, we have
\begin{equation}\label{kappav-eqn4}
[\kappa+v, x]=(ad_v-ad_{\kappa}^*)x.
\end{equation}
As a consequence, we get that $ad_{\kappa+v}(\g)=0$  if and only if $ad_{\kappa}^*=ad_v$.
\end{enumerate}
\end{proof}

\subsection{The Atiyah class of $(L=\g\Join\g^*=sl(n,\C), A=\g=sb(n,\C), E=\g^*=su(n))$}

As shown in Example \ref{Lu-Weinstein-exa}, $(L=\g\Join\g^*=sl(n,\C), \g=sb(n,\C),\g^*=su(n))$ is a Manin triple. The non-degenerate bi-linear form on $L=sl(n,\C)$ 
is defined by 
\begin{equation}
\langle X,Y\rangle=Im(trace(XY))
\end{equation}
 for all $X,Y\in sl(n,\C)$. 

Let $\mathfrak{t}$ be the subspace of $\g=sb(n,\C)$ consisting of all $n\times n$ real diagonal traceless matrices. Let $\mathfrak{n}_{+}$ be the subspace of $\g=sb(n,\C)$ consisting of all $n\times n$ strictly upper triangular matrices. Then we have 
$$\g=\mathfrak{t}\oplus\mathfrak{n}_{+}.$$
Moreover, $\mathfrak{h}=\mathfrak{t}\oplus\sqrt{-1}\mathfrak{t}$ is
 a Cartan subalgebra of $L=\g\Join\g^*=sl(n,\C)$, 
where $\mathfrak{t}\subset\g=sb(n,\C)$ and $\sqrt{-1}\mathfrak{t}\subset\g^*=su(n)$.

\begin{lemma}\label{kappa-sbsu}
Let $(\g=sb(n,\C),\g^*=su(n))$ be the Lie bialgebra as in Example \ref{Lu-Weinstein-exa}.
Let $\kappa\in\g^*$ be defined by Equation \eqref{kappa-eqn}. 
Then we have 
\begin{enumerate}
\item $\kappa\neq0$,
\item $\kappa\in\sqrt{-1}\mathfrak{t}$.
\end{enumerate}
\end{lemma}
\begin{proof}
\begin{enumerate}
\item
For any $t\in\mathfrak{t}$, we have 
\begin{equation}
\langle\kappa,t\rangle=trace(ad_t)=\sum_{\alpha\in\Delta_{+}}\langle\alpha,t\rangle
=\langle\sum_{\alpha\in\Delta_{+}}\alpha,t\rangle,
\end{equation}
where $\Delta_{+}\subset\mathfrak{h}^*$ is the set of positive roots for the Cartan subalgebra 
$\mathfrak{h}=\mathfrak{t}\oplus\sqrt{-1}\mathfrak{t}$. 
Since $\sum_{\alpha\in\Delta_{+}}\alpha$ is a nonzero vector in $\mathfrak{h}^*$, 
we get $k\neq 0$.   

\item
For any $y\in\mathfrak{n}_{+}$ and $\xi\in\sqrt{-1}\mathfrak{t}$, $y$ and $\xi$ are orthogonal under the Killing form. Hence we have
$$\langle y,\xi\rangle=Im(trace(y\cdot\xi))=0,$$
which implies that $$\sqrt{-1}\mathfrak{t}\subset\mathfrak{n}_{+}^{\perp},$$
where $\mathfrak{n}_{+}^{\perp}$ denotes by the annihilator of 
$\mathfrak{n}_{+}$ in $\g^*$.
As $$\dim\sqrt{-1}\mathfrak{t}=\dim\mathfrak{t}=\dim\g-\dim\mathfrak{n}_{+},$$ 
we get that
\begin{equation}
\sqrt{-1}\mathfrak{t}=\mathfrak{n}_{+}^{\perp}.
\end{equation}
On the other hand,  $ad_y\in End(\g)$ is nilpotent for all $y\in\mathfrak{n}_{+}$.
It implies
$$\kappa(y)=trace(ad_y)=0$$
for all $y\in\mathfrak{n}_{+}$. Thus we have
$$\kappa\in\mathfrak{n}_{+}^{\perp}=\sqrt{-1}\mathfrak{t}.$$
\end{enumerate}
\end{proof}

\begin{proposition}\label{LW-prop}
The first scalar Atiyah class  $c_1(E)$ associated with the triple 
$(L=sl(n,\C), A=\g=sb(n,\C),E=\g^*=su(n))$ does not vanish. 
As a consequence, the Atiyah class $\alpha_E$ associated with the triple 
$(L=sl(n,\C), A=\g=sb(n,\C),E=\g^*=su(n))$ does not vanish.
\end{proposition}

\begin{proof}
Assume that the first scalar Atiyah class  $c_1(E)$ associated with the triple 
$(L=sl(n,\C), A=\g=sb(n,\C),E=\g^*=su(n))$ vanishes.
Then by Theorem \ref{c1-thm}, there exists $v\in\g$ such that 
\begin{equation}\label{kappav-prop-eqn1}
ad_{\kappa+v}(\g)=0.
\end{equation}
As $\g=\mathfrak{t}\oplus\mathfrak{n}_{+}$, $v\in\g$ can be written as $$v=v_1+v_2,$$
where $v_1\in\mathfrak{t}$ and $v_2\in\mathfrak{n}_{+}$. 
By Equation \eqref{kappav-prop-eqn1}, we have
$$[\kappa+v,v_2]=0,$$ which implies 
\begin{equation}\label{kappav-prop-eqn2}
[\kappa+v_1,v_2]=[\kappa+v_1+v_2, v_2]=[\kappa+v, v_2]=0.
\end{equation}
By Lemma \ref{kappa-sbsu}, we get $\kappa\in\sqrt{-1}\mathfrak{t}.$

Since $\kappa+v_1\in\mathfrak{h}=\mathfrak{t}\oplus\sqrt{-1}\mathfrak{t}$ and $v_2\in\mathfrak{n}_{+}$, by Equation \eqref{kappav-prop-eqn2},
$\kappa+v\in sl(n,\C)$ has the Jordan decomposition
\begin{equation}
\kappa+v=(\kappa+v_1)+v_2,
\end{equation}
where $\kappa+v_1$ is the semisimple part and $v_2$ is the nilpotent part.
As a consequence, we get the Jordan decomposition
\begin{equation}
ad_{\kappa+v}=ad_{\kappa+v_1}+ad_{v_2},
\end{equation}
where $ad_{\kappa+v_1}\in End(L)$ is the semisimple part, 
$ad_{v_2}\in End(L)$ is the nilpotent part.
By Equation \eqref{kappav-prop-eqn1}, we obtain
\begin{equation}
ad_{\kappa+v_1}(\g)=0,
\end{equation}
which implies
$$ad_{\kappa+v_1}(\mathfrak{n}_{+})=0.$$ 
Therefore we have
\begin{equation}\label{kappa-prop-eqn3}
\langle\kappa+v_1,\alpha\rangle=0
\end{equation}
for all $\alpha\in\Delta_{+}\subset\mathfrak{h}^*$.
As a consequence of Equation \eqref{kappa-prop-eqn3}, we get
$$\kappa+v_1=0.$$ 
Since $\kappa\in\mathfrak{t}$ and $v_1\in\sqrt{-1}\mathfrak{t}$, we obtain 
$$\kappa=0,$$
which contradicts Lemma \ref{kappa-sbsu}.

Thus $c_1(E)$ does not vanish. And consequently, the Atiyah class $\alpha_E$ does not vanish.
\end{proof}

\end{document}